\documentclass[12pt,reqno]{article}
\usepackage{amssymb,amscd,amsmath,amsthm,color}
\newtheorem{theorem}{Theorem}[section]
\newtheorem{lemma}[theorem]{Lemma}
\newtheorem{cor}[theorem]{Corollary}
\newtheorem{prop}[theorem]{Proposition}
\newtheorem{conjecture}[theorem]{Conjecture}
\usepackage{enumerate,amssymb}
\theoremstyle{definition}

\newtheorem{example}[theorem]{Example}
\newtheorem{question}[theorem]{Question}
\theoremstyle{remark}
\newtheorem{remark}[theorem]{Remark}

\numberwithin{equation}{section}

\def\ind{\mathbf{1}}

\def\bM{\mathbb{M}}
\def\bN{\mathbb{N}}

\newcommand{\sym}{\mathrm{sym}}
\newcommand{\qsym}{\mathrm{qsym}}

\begin{document}
\baselineskip=15pt

\title{Averages over matrix unitary orbits \\ and  spectral order}

\author{ Jean-Christophe Bourin and Eun-Young Lee}

\date{ }

\maketitle

\vskip 10pt\noindent
{\small
{\bf Abstract.} We establish  matrix versions of the comparisons between the $\ell^p$-norms or quasi-norms for sequences of complex numbers.
For instance, given $1\ge q>0$, and a family of $m$ normal $d\times d$ matrices $A_1,\ldots, A_m$, we show that
$$
\left|\sum_{k=1}^m A_k\right| \le \frac{1}{d}\sum_{i=1}^d V_i\left\{\sum_{k=1}^m  |A_k|^{q}\right\}^{1/q}\!\!\!\!V_i^*
$$
for some unitary $d\times d$ matrices $V_1,\ldots, V_d$.  We also give  applications  to Olson's spectral order  and to the comparison between  the symmetric modulus and the  quadratic symmetric modulus. In particular we show that the sum $A+B$ of two positive matrices submajorizes their Kato supremum $A\vee B$, thereby completing majorization results due to Ando.

}

\section{Introduction}

A basic inequality for $m$ complex numbers $a_1,\ldots,  a_m$ is the comparison, for $0<q\le 1\le  p$, between the $\ell^p$-norm and the $\ell^q$-quasi-norm,
\begin{equation}\label{elem}
 \left\{\sum_{k=1}^m |a_k|^p\right\}^{1/p} \le  \sum_{k=1}^m |a_k| \le \left\{\sum_{k=1}^m |a_k|^q\right\}^{1/q} .
\end{equation}

What can be said in the matrix setting? Investigating this question has led us to several elegant matrix inequalities involving unitary orbits and convex or concave functions. These results are established in the next section, where we also discuss an application to the quadratic symmetric modulus arising in Zhang’s triangle inequality. We further explain the significance of Zhang’s 2026 result.

The third section is devoted to applications to the spectral order, including comparisons between sums and parallel sums  of positive matrices and their suprema and infima. This completes a line of majorisation results initiated by Ando in the late 1980s. We also include a concise treatment of the spectral order relation $\preceq$, together  with a new proof of the fact that, for positive operators,
$$
A\preceq B \iff  A^n\le B^n \ \ \! {\mathrm{ for\ all}}\ n\in\bN.
$$
This  order, introduced by Olson in 1971, endows the set of Hermitian operators with a lattice structure,
extending the lattice of projections. It allows us  to introduce the maximal and minimal symmetric moduli. Adapting Zhang’s inequality to these moduli
 appears to be a challenging problem.

Our main tools are powerful matrix analogues of scalar inequalities for convex and concave functions. One of the earliest examples along these lines is von Neumann's trace inequality (the special case $g(t)=t\log t$ is  the concavity of von Neumann's entropy), which states that
\begin{equation}\label{vN}
{\mathrm{Tr\,}} g\left(\frac{A+B}{2}\right) \le {\mathrm{Tr\,}}\frac{ g(A)+g(B)}{2} 
\end{equation}
for every convex function $g(t)$ defined on an interval containing the spectra of the Hermitian matrices $A$ and $B$. If furthermore $g(0)$ =0 and $A$ and $B$ are positive semidefinite, then
\begin{equation}\label{Rot}
{\mathrm{Tr\,}} g(A+B) \ge {\mathrm{Tr\,}}\! \left(g(A)+ g(B) \right)
\end{equation}
as shown by Rotfel'd in 1969 \cite{Rot}. We shall use   refinements of \eqref{vN}-\eqref{Rot} involving averages over the unitary orbit of a Hermitian matrix (or over the orthogonal orbit in the real symmetric case).  In a short final section, we establish a relation between a real symmetric matrix $S$ and its diagonal part $\Delta(S)$. In the $3\times3$ case,
$$
\Delta(S)=\frac{1}{4}\sum_{i=1}^4 U_iSU_i^*
$$
for some orthogonal matrices $U_1,\ldots, U_4$. We shall employ this  observation   to derive real analogues of results such as the inequality stated in the Abstract.

Thus, this  note lies within the longstanding tradition in matrix analysis of extending scalar  inequalities and identities to the noncommutative setting of matrices. Before proceeding, 
we  fix some notation.  We denote by $\bM_d$  the space of complex $d\times d$ matrices and $\bM_d^+$  its cone of positive semidefinite matrices. For $A\in\bM_d$, we write $|A|=(A^*A)^{1/2}$  for its absolute value (or right modulus).  If $A\in\bM_d$ is  Hermitian, we denote by 
$$\lambda_1^{\downarrow}(A)\ge \cdots\ge  \lambda_d^{\downarrow}(A)$$  its eigenvalues arranged  in   nonincreasing order, and by
$
A^{\downarrow}$ the diagonal matrix in $\bM_d$ having these eigenvalues on the diagonal. For a pair of Hermitian matrices $A,B\in\bM_d$, the weak majorization, or submajorization, relation $A\prec_w B$ means that
 \begin{equation}\label{submajdef}
\sum_{j=1}^k\lambda_j^{\downarrow}(A) \le  \sum_{j=1}^k\lambda_j^{\downarrow}(B)
\end{equation}
for all $k=1,\ldots, d$. If,
 furthermore,   equality holds for $k=d$, that  is, if $A$ and $B$ have  the same trace, then we say that $A$  is majorized by $B$, and we write $A\prec B$. For each $k$, the  sum of the $k$ largest eigenvalues occuring in \eqref{submajdef} is a subbaditive functional on the Hermitian part of $\bM_d$.

\section{Unitary orbits }

A stronger form of  \eqref{vN} states that, for Hermitian matrices $A,B\in\bM_d$ and a convex function $g(t)$ defined on an interval containing the spectra of $A$ and $B$, 
\begin{equation}\label{vNB}
 g\left(\frac{A+B}{2}\right) \le \frac{1}{2}\left\{U \frac{ g(A)+g(B)}{2} U^* + V \frac{ g(A)+g(B)}{2} V^*\right\}
\end{equation}
for some unitary matrices $U,V\in\bM_d$, \cite[Corollary 2.2]{BL}. For a concave function, the inequality reverses. 
We shall use a multivariable version of this result, together with related inequalities, to prove our main theorems.

\subsection{Convex and concave functions}

\vskip 5pt
\begin{theorem}\label{th1} Let $g(t)$ and $h(t)$ be two nondecreasing convex functions on $[0,\infty)$ such that $h(0)=0$. If $\{A_k\}_{k=1}^m$ are  matrices in $\bM_d^+$, then there exist unitary matrices $\{V_i\}_{i=1}^d$  in $\bM_d$ such that 
$$ 
g\left(\sum_{k=1}^m h(A_k)\right) \le \frac{1}{d} \sum_{i=1}^d V_i\left\{ g\circ h\left(\sum_{k=1}^m A_k\right)\right\}V_i^*.
$$
\end{theorem}

\vskip 10pt
As a special case, taking $g(t)=t^r$ and $h(t)=t^{p}$
 with $p,r\ge 1$ we obtain a matrix version of the left-hand inequality in \eqref{elem}.

\vskip 5pt
\begin{cor} Let $\{A_k\}_{k=1}^m$ be matrices in $\bM_d^+$ and let $p,r\ge 1$. Then there exist unitary matrices $\{V_i\}_{i=1}^d$  in $\bM_d$ such that 
$$
\left(\sum_{k=1}^m A_k^p\right)^{r}  \le \frac{1}{d} \sum_{i=1}^d V_i\left\{\sum_{k=1}^m A_k\right\}^{pr}\!V_i^*.
$$
\end{cor}

\vskip 5pt\noindent
\begin{proof} (Theorem \ref{th1}) From \cite[Corollary 3.2]{BLjot} we have
\begin{equation}\label{jot}
\left\|  \sum_{k=1}^m h(A_k) \right\|\le \left\|  h\left(\sum_{k=1}^m A_k\right) \right\| 
\end{equation}
for all symmetric (i.e., unitarily invariant) norms $\|\cdot\|$ on $\bM_d$. Let 
\begin{equation*}
A:=\sum_{k=1}^m h(A_k), \quad B:= h\left(\sum_{k=1}^m A_k\right).
\end{equation*} 
By the Fan dominance principle, \eqref{jot} is equivalent to   $$A\prec_{w}B.$$
Consider the spectral decomposition
$$
B=\sum_{k=1}^d\lambda_k^{\downarrow}(B)E_k,
$$
and define
$$
B':= \left(\sum_{k=1}^{d-1}\lambda_k^{\downarrow}(B)E_k\right) + \mu E_d
$$
where $\mu:= \lambda_d^{\downarrow}(B)-{\mathrm{Tr\,}} B +{\mathrm{Tr\,}} A $  is chosen so that
$$
A\prec B'.
$$
By the Schur-Horn theorem
(see \cite{M} or \cite{Chan} for  elegant proofs), $A$ is unitarily equivalent to the diagonal part 
 of $WB'W^*$ for some unitary $W$; that is,
\begin{equation}\label{SH}
A=U\Delta(WB'W^*)U^*
\end{equation}
for some unitary $U\in\bM_d$.
By an identity due to Bhatia \cite[Eq.\ (2)]{Bh-monthly}, there exist diagonal
unitary matrices $R_1,\ldots,R_d\in\bM_d$ such that
\[
\Delta(X)=\frac{1}{d}\sum_{i=1}^d R_iXR_i^*
\]
for every matrix $X\in\bM_d$. Applying this identity to
$X=WB'W^*$ yields
\[
\Delta(WB'W^*)
=
\frac{1}{d}\sum_{i=1}^d
R_iWB'W^*R_i^*.
\]
Hence,
\[
A
=
U\Delta(WB'W^*)U^*
=
\frac{1}{d}\sum_{i=1}^d
(UR_iW)\,B'\,(UR_iW)^*.
\]
Setting $U_i:=UR_iW$, we obtain
\[
A=\frac{1}{d}\sum_{i=1}^d U_iB'U_i^*.
\]
Since $g(t)$ is nondecreasing, we have a unitary $U_0$ such that
$$
g(A)\le U_0g\left(\frac{1}{d}\sum_{i=1}^d U_i BU_i^*\right)U_0^*.
$$
Now, we can take $U=V$ in \eqref{vNB} when the convex function is  monotone, and  a related inequality for $m$ matrices \cite[Corollary 2.4]{BL} ensures that
$$
g\left(\frac{1}{d}\sum_{i=1}^d U_i BU_i^*\right) \le V_0 \left(\frac{1}{d}\sum_{i=1}^d U_i g(B)U_i^*\right)V_0^*.
$$
This combined with the previous inequality yields
$$
g(A) \le \frac{1}{d}\sum_{i=1}^d V_i g(B)V_i^*
$$
for some unitaries $V_1,\ldots, V_d$. Returning to  the definitions of $A$ and $B$ completes the proof. 
\end{proof}

\vskip 5pt
We record a key point of this proof in the following lemma. The notation $\bM_d(\Omega)$
stands for the Hermitian part of $\bM_d$ with spectra in an interval $\Omega$ of the real line.

\vskip 5pt
\begin{lemma}\label{L1} Let $A,B\in\bM_d(\Omega)$ and let $g(t)$ be a nondecreasing convex function on $\Omega$. If $A\prec_w B$, then 
there exist unitary matrices $\{V_i\}_{i=1}^d$  in $\bM_d$ such that 
$$
g(A)\le \frac{1}{d}\sum_{i=1}^d V_i g(B)V_i^*.
$$
\end{lemma}

\vskip 5pt
Beyond its role in the proof of Theorem \ref{th1}, Lemma \ref{L1} entails a fundamental property of submajorization: convex nondecreasing functions $g(t)$ on $\Omega$ preserve submajorization $\prec_w$ on $\bM_d(\Omega)$:
\begin{equation}\label{fund}
A\prec_w  B \Rightarrow g(A) \prec_w g(B).
\end{equation}
Our proof of Lemma \ref{L1} (and hence of Theorem \ref{th1}) does not depend on \eqref{fund}  and provides an alternative derivation of \eqref{fund} from a unitary Jensen inequality combined with the Schur-Horn theorem.

\vskip 5pt
\begin{theorem}\label{th2} Let $f,g:[0,\infty)\to[0,\infty)$ be nondecreasing. Assume that $f(t)$ is concave and  $g(t)$ is convex. If $\{A_k\}_{k=1}^m$ are  normal matrices in $\bM_d$, then there exist unitary matrices $\{V_i\}_{i=1}^d$  in $\bM_d$ such that 
$$ 
g\circ f\left(\left|\sum_{k=1}^m A_k\right|\right) \le \frac{1}{d} \sum_{i=1}^d V_i g\left(\sum_{k=1}^m f(|A_k|)\right)V_i^*.
$$
\end{theorem}

\vskip 5pt
\begin{proof} It was shown in \cite{Bpams} that, for every pair of normal matrices $N,M\in\bM_d$, we have 
$$
\| f(|N+M|)\| \le \| f(|N|) + f(|M|) \|
$$
for all symmetric norms.  The corresponding $m$-variables version  holds,
  as stated in \cite[Theorem 3.3]{Bpams}, or just before \cite[Corollary 2.4]{Bpams}. Thus
$$
\left\|  f\left(\left|\sum_{k=1}^m A_k\right|\right)  \right\| \le \left\| \sum_{k=1}^m f(|A_k|)\right\|.
$$
This is a submajorization relation $\prec_w$ between the matrix in the left-hand and that one in the right-hand. So we may apply Lemma \ref{L1}
to conclude
\end{proof}

\vskip 5pt
For $1\ge q>0$ and  the functions $g(t)=t^{1/q}$ and $f(t)=t^{q}$,  we  obtain a matrix version of the right-hand inequality in \eqref{elem}:

\vskip 5pt
\begin{cor}\label{corq} Let $\{A_k\}_{k=1}^m$ be normal matrices in $\bM_d$ and let $1\ge q>0$. Then there exist unitary matrices $\{V_i\}_{i=1}^d$  in $\bM_d$ such that 
$$
\left|\sum_{k=1}^m A_k \right| \le \frac{1}{d} \sum_{i=1}^d V_i\left\{\sum_{k=1}^m |A_k|^q\right\}^{1/q}\!\!\!\!V_i^*.
$$
\end{cor}

\vskip 5pt
Let ${\mathrm{Re\,}}Z=(Z+Z^*)/2$ and ${\mathrm{Im\,}}Z=(Z-Z^*)/2i$ denote the real and imaginary parts of $Z\in\bM_d$. Two
further special cases of Theorem \ref{th2} are the following.

\vskip 5pt
\begin{cor}\label{cartpq} Let $Z\in\bM_d$  and let $p\ge 1\ge q>0$. Then there exist unitary matrices $\{V_i\}_{i=1}^d$  in $\bM_d$ such that 
$$
\left| Z\right|^{pq} \le \frac{1}{d} \sum_{i=1}^d V_i\left(|{\mathrm{Re\,}}Z|^q+ |{\mathrm{Im\,}}Z|^q\right)^{p}V_i^*.
$$
\end{cor}

\vskip 5pt
\begin{cor} Let $Z\in\bM_d$  and let $c>0$. Then there exist unitary matrices $\{V_i\}_{i=1}^d$  in $\bM_d$ such that 
$$
\left| Z\right|\wedge cI \le \frac{1}{d} \sum_{i=1}^d V_i\left(|{\mathrm{Re\,}}Z|\wedge cI +  |{\mathrm{Im\,}}Z|\wedge cI\right)V_i^*.
$$
\end{cor}

\vskip 5pt
Here, if $T\in\bM_d^+$, $$T\wedge cI=m_c(T), \qquad   m_c(t)=\min\{t,c\},$$ and $m_c(t)$  is nonnegative and concave on $[0,\infty)$. We shall consider the infimum $A\wedge B$ of general matrices $A,B\in \bM_d^+$ in Section 3.

A companion result of Corollary \ref{corq} is:
\vskip 5pt
\begin{prop} Let $\{A_k\}_{k=1}^m$ be normal matrices in $\bM_d$ and $\alpha>0$. Then there exist unitary matrices $\{V_i\}_{i=1}^d$  in $\bM_d$ such that 
$$
\left|\sum_{k=1}^m A_k \right|^{\alpha} \le \frac{1}{d} \sum_{i=1}^d V_i\left\{\sum_{k=1}^m |A_k|\right\}^{\alpha}\!\!V_i^*.
$$
\end{prop}

\vskip 5pt
\begin{proof} From  \cite[Corollary 2.10]{BLtv} we have 
$$
\left|\sum_{k=1}^m A_k \right|^{\alpha}\prec_w \left\{\sum_{k=1}^m |A_k|\right\}^{\alpha}.
$$
Lemma \ref{L1} with $g(t)=t$
 completes the proof.
\end{proof}

\subsection{Symmetric moduli}

For a matrix $Z\in\bM_d$, its left and right moduli $|Z^*|$ and $|Z|$ are the positive part in the polar decompositions
$$
Z=|Z^*|U=|Z^*|^{1/2}U|Z|^{1/2}=U|Z|.
$$
These moduli occur in a number of important operator inequalities. There is no reason to privilege the left or the right modulus. Hence, the consideration of symmetrized versions, or means of these two moduli, should bring several interesting results. 
The symmetric modulus and the quadratic symmetric modulus of $Z\in\bM_d$ are respectively defined as
$$
|Z|_{\sym}:= \frac{|Z|+|Z^*|}{2}, \qquad
|Z|_{\qsym}:=\sqrt{\frac{|Z|^2+|Z^*|^2}{2}}.
$$
Note that $$|Z|_{\qsym}=\sqrt{|{\mathrm{Re\,}}Z|^2+ |{\mathrm{Im\,}}Z|^2}$$
and
 \begin{equation}\label{compsqs}
|Z|_{\sym}\le |Z|_{\qsym}
\end{equation} 
by operator concavity of $\sqrt{t}$. Very recently Teng Zhang obtained a    triangle  inequality for the quadratic modulus \cite[Theorem 1.9]{TZsym}:

\vskip 5pt\noindent
\begin{theorem}\label{thZ}  Let $A,B\in\bM_d$. Then, for some unitaries $U,V\in\bM_d$,
$$
|A+B|_{\qsym} \le U|A|_{\qsym}U^* +V|B|_{\qsym}V^*.
$$
\end{theorem}

\vskip 5pt
This  is a symmetrized version of the famous Thompson's inequality \cite{T}.  Zhang's result
    was motivated by  the following  special case   \cite{BLsym}:

\vskip 5pt
\begin{cor}\label{symq} Let $Z\in\bM_d$. Then, for some unitaries $U,V\in\bM_d$,
$$
 |Z|_{\qsym} \le U |{\mathrm{Re\,}}Z| U^* + V |{\mathrm{Im\,}}Z|V^*.
$$
\end{cor}

Here \eqref{compsqs} shows that we may replace the left-hand by $|Z|_{\sym}$. 
Despite  this fact  Theorem \ref{thZ} does not hold for  the symmetric modulus $|\cdot|_{\sym}$.   Two matrices $A,B\in\bM_3$ are given in \cite{Z2} with $$\lambda_1^{\downarrow}(|A+B|_{\sym})=\sqrt{2} \left\{ \lambda_1^{\downarrow}(|A|_{\sym})+\lambda_1^{\downarrow}( |B|_{\sym})\right\}.$$ Here the constant $\sqrt{2}$ cannot be augmented;  a result of \cite{BLtriang} states a weak form of the triangle inequality as follows.

\vskip 5pt
\begin{theorem}\label{cornor}   Let $\{X_k\}_{k=1}^m$ be in $\bM_d$. Then, 
$$
\left|\sum_{k=1}^m X_k\right|_{\sym} \prec_w \sqrt{2} \sum_{k=1}^m \left|X_k\right|_{\sym}.
$$
\end{theorem}

\vskip 5pt 
We can use Theorem \ref{th1} to get a  further comparison than \eqref{compsqs}.
The next result shows that the two symmetric moduli are equivalent up to the optimal constant 
$\sqrt{2}$ 
 in the unitary-orbit order.

\vskip 5pt
\begin{cor} Let $Z\in \bM_d$. Then there exist unitary matrices $\{V_i\}_{i=1}^d$  in $\bM_d$ such that 
$$
|Z|_{\qsym} \le \frac{ \sqrt{2}}{d}  \sum_{i=1}^d V_i |Z|_{\sym}V_i^*.
$$
The constant $\sqrt{2}$ cannot be diminished.
\end{cor}

\vskip 5pt
\begin{proof} To see that $\sqrt{2}$ is optimal, it suffices to pick
$$
Z=\begin{pmatrix} 0&2 \\ 0&0\end{pmatrix}
$$
as
$$
|Z|_{\qsym} =  \begin{pmatrix} \sqrt{2}&0 \\ 0&\sqrt{2}\end{pmatrix}, \qquad |Z|_{\sym} =  \begin{pmatrix} 1&0 \\ 0&1\end{pmatrix}.
$$
To establish the main conclusion, apply Theorem \ref{th2} with two matrices $A_1=|Z|^2$, $A_2=|Z^*|^2$, and the functions $f(t)=\sqrt{t}$ and $g(t)=t$. This gives
$$
\sqrt{|Z|^2+|Z^*|^2} \le \frac{ 1}{d}  \sum_{i=1}^d V_i \left(|Z|+|Z^*|\right)V_i^*
$$
and dividing both sides by $\sqrt{2}$ yields the result.
\end{proof}

\subsection{Supermajorization}

Let $A,B\in\bM_d(\Omega)$. The supermajorization relation $A\prec ^{w} B$ means that $-A\prec_w -B$. In other words,
$$
\sum_{j=1}^k\lambda_j^{\uparrow}(A) \ge  \sum_{j=1}^k\lambda_j^{\uparrow}(B)
$$
for all $k=1,\ldots, d$, where $\lambda_k^{\uparrow}(\cdot)$ stand for the eigenvalues arranged in nondecreasing order. We then have the following form of Lemma \ref{L1}.

\vskip 5pt
\begin{lemma}\label{L2} Let $A,B\in\bM_d(\Omega)$ satisfy $A\prec^w B$ and let $f(t)$ be a nondecreasing concave function on $\Omega$. Then 
there exist unitary matrices $\{V_i\}_{i=1}^d$  in $\bM_d$ such that 
$$
f(A)\ge \frac{1}{d}\sum_{i=1}^d V_i f(B)V_i^*.
$$
\end{lemma}

\begin{proof} Since $-A\prec_w -B$, we may apply Lemma \ref{L1} to $-A$, $-B$, and the nondecreasing convex function $-f(-t)$ on $-\Omega$.
This proves the lemma.
\end{proof}

\section{Spectral order}

Until now we have dealt with the usual order $\le$ on the Hermitian part of $\bM_d$. Another remarkable order is the spectral order $\preceq$, introduced by Olson in 1971 \cite{Ols} to endow the Hermitian operators on a Hilbert space with a lattice structure extending that of projections.

\subsection{Basic properties }

We confine ourselves to the setting of $\bM_d^+$ and   motivate the spectral order $\preceq$ as an order on $\bM_d^+$  such that $A\preceq B$ is equivalent  to the natural functional monotonicity property $(ii)$ below. Olson proved the following proposition.

\begin{prop}\label{OO}
 For  $A,B\in\bM_d^+$, the following conditions are equivalent:
\begin{itemize}
\item[($i$)] $\ind_{(x,\infty)}(A) \le  \ind_{(x,\infty)}(B)$ for all $x> 0$;
\item[($ii$)] $f(A)\le f(B)$ for all nondecreasing functions $f:[0,\infty)\to[0,\infty)$.
\item[($iii$)] $A^n\le B^n$ for all integers $n=1,2,\ldots$.
\end{itemize}
When these conditions hold, $A$ is spectrally dominated by $B$, and we write $A\preceq B$.
\end{prop}

\vskip 5pt
Another interesting equivalent condition can be found in \cite{HJS}. 

The implication 
$(i)\Rightarrow(ii)$ is easy as we may assume that $f(t)=c_0+\sum_{i=1}^n c_i\ind_{(x_i,\infty)}$ for some positive numbers $x_1,\ldots, x_n$ and $c_0,\ldots, c_n$. $(ii)\Rightarrow(iii)$ is trivial.

We give a new, short  proof of the nontrivial implication $(iii)\Rightarrow (i)$. Suppose $(iii)$ holds and  $B$ is invertible (and so is $A$).  Since $1/t$ is operator decreasing, ($iii$) ensures
$$
\frac{1}{A^{-1}+\cdots+A^{-n}} \le \frac{1}{B^{-1}+\cdots+B^{-n}}.
$$
Since $1/(t^{-1}+\cdots+ t^{-n})$ pointwise converges to $(t-1)_+:=\max\{0, t-1\}$ as $n\to\infty$, we infer that
$$
(A-1)_+ \le (B-1)_+.
$$
Therefore, for all $0<q<1$,
$$
(A-1)^q_+ \le (B-1)^q_+,
$$
and letting $q\searrow 0$ we obtain $\ind_{(1,\infty)}(A)\le \ind_{(1,\infty)}(B)$. Replacing $A$ and $B$ by $x^{-1}A$ and $x^{-1}B$ then gives the condition ($i$) for all $x\ge 0$ when $A$ and $B$ are invertible. The general case follows as $A^n\le B^n$ entails $(A+ I)^n\le (B+ I)^n$ because $(x^{1/n}+1)^n$ is operator monotone. Thus $(iii)$ ensures $(i)$ for $A+ I$ and  $B+ I$, which is equivalent to $(i)$ for $A$ and $B$ since $\ind_{(x,\infty)}(A+I)=
\ind_{(x-1,\infty)}(A)$.

\vskip 5pt
\begin{remark} Let $A,B\in\bM_d^+$ be invertible. Condition $(iii)$ in Proposition \ref{OO} shows that
$$
A\preceq B \iff B^{-1}\preceq A^{-1},
$$
and, in general, $A\preceq A+B$ does not hold.
\end{remark}

\vskip 5pt
\begin{remark} If $A,B\in\bM_d^+$, then the family of projections $\ind_{(x,\infty)}(A) \vee  \ind_{(x,\infty)}(B)$ is nonincreasing with respect to $x>0$. Hence there exists a unique $C\in\bM_d^+$ such that
$$
\ind_{(x,\infty)}(C)=\ind_{(x,\infty)}(A) \vee  \ind_{(x,\infty)}(B).
$$
Condition $(i)$ in Proposition \ref{OO} implies that $A\preceq C$ and $B\preceq C$, while any $X\in\mathbb M_d^+$ satisfying $A\preceq X$ and $B\preceq X$ must also satisfy $C\preceq X$.
 Consequently, $C=:A\vee  B$ is the least upper bound of $A$ and $B$. One defines in a similar way the greatest lower bound $D=:A\wedge B$ with
$$
\ind_{(x,\infty)}(D)=\ind_{(x,\infty)}(A) \wedge  \ind_{(x,\infty)}(B).
$$
Therefore $(M_d^+,\preceq, \vee,\wedge)$ is a lattice.
\end{remark}

\vskip 5pt
\begin{example}\label{funny} Consider two matrices in $\bM_2^+$
$$
A=\begin{pmatrix} 1&0 \\ 0&0
\end{pmatrix},
\qquad
B=\begin{pmatrix} 1&1 \\ 1&1
\end{pmatrix}.
$$
Then their supremum $A\vee B$ is characterized by:

\noindent
1) If $0<x\le 1$, then
$$
\ind_{(x,\infty)}(A\vee B) = \begin{pmatrix} 1&0 \\ 0&0
\end{pmatrix}\vee \begin{pmatrix} 1/2&1/2 \\ 1/2&1/2
\end{pmatrix}=
\begin{pmatrix} 1&0 \\ 0&1
\end{pmatrix}.
$$
2) If $1<x<2$, then
$$
\ind_{(x,\infty)}(A\vee B) = \begin{pmatrix} 0&0 \\ 0&0
\end{pmatrix}\vee \begin{pmatrix} 1/2&1/2 \\ 1/2&1/2
\end{pmatrix}=
 \begin{pmatrix} 1/2&1/2 \\ 1/2&1/2
\end{pmatrix}.
$$
3) If $2\le x$, then $\ind_{(x,\infty)}(A\vee B) =0$.

Therefore
$$
A\vee B =  1I+(2-1)\begin{pmatrix} 1/2&1/2 \\ 1/2&1/2
\end{pmatrix}= \begin{pmatrix} 3/2&1/2 \\ 1/2&3/2\end{pmatrix}.
$$
Hence, in this example,
$
{\mathrm{Tr\,}} A\vee B= {\mathrm{Tr\,}} (A+B)
$,
this seems not surprizing as $A$ and $B$ are scalar multiples  of two noncommuting rank one projections.
\end{example}

\vskip 5pt
\begin{question} In $\bM_d^+$, does the equivalence
$
{\mathrm{Tr\,}} A\vee B= {\mathrm{Tr\,}} (A+B) \iff A\wedge B=0
$
hold ? Note that $A\wedge B=0$ means that  the ranges of $A$ and $B$ have their intersection reduced to $\{0\}$.
\end{question}

\subsection{Sum and supremum}

For two positive numbers, $a\vee b=\lim_{n\to\infty}(a^n+b^n)^{1/n}$. In 1979, Kato \cite{K} obtained the following matrix analogue.

\vskip 5pt
\begin{prop}  Let $A,B\in\bM_d^+$. Then 
$$
A\vee B =\lim_{n\to\infty} \left(A^n+B^n\right)^{1/n}.
$$
\end{prop}

An interesting generalization is \cite[Theorem 9]{AW}. For convenience of the reader we prove Kato's limit theorem.

\vskip 5pt
\begin{proof} Let $m>n\ge 1$. Operator concavity of $t\mapsto t^{n/m}$ ensures
$$
\frac{A^n+B^n}{2} \le \left(\frac{A^m+B^m}{2}\right)^{n/m}.
$$
As $t\mapsto t^{1/n}$
is operator monotone,
$$
\left(\frac{A^n+B^n}{2}\right)^{1/n} \le \left(\frac{A^m+B^m}{2}\right)^{1/m}.
$$
Hence we have  a bounded nondecreasing sequence with respect to the usual order in $\bM_d^+$. Since $2^{1/n}\to 1$, its limit is $\lim_{n\to\infty}(A^n+B^n)^{1/n}$. Fix an integer $k>0$. For $n>k$,
$
A^k, B^k \le (A^n+B^n)^{k/n} $  because $A^n\le A^n+B^n$ and $t^{k/n}$ is operator monotone.
Thus
$$
A^k, B^k \le \left(\lim_{n\to\infty}(A^n+B^n)^{1/n}\right)^k$$
Hence $(iii)$ in Proposition \ref{OO} implies
$$
A,B \preceq \lim_{n\to\infty}(A^n+B^n)^{1/n}.$$

 On the other hand, suppose $A,B\preceq X$ and fix $m\in\bN$. By $(iii)$ in Proposition \ref{OO}, for every integer $n$,
\begin{equation}\label{convOO}
 \frac{A^n+B^n}{2} \le X^n.
\end{equation}
If further $n\ge m$,
$$
 \left(\frac{A^n+B^n}{2}\right)^{m/n} \le X^m
$$
as  $t^{m/n}$ is operator monotone. Letting $n\to \infty$ yields
$$
\left\{\lim_{n\to\infty } \left(\frac{A^n+B^n}{2}\right)^{1/n}\right\}^m \le X^m
$$
for every integer $m$. So, again by this crucial condition $(iii)$, 
 $$
\lim_{n\to\infty } \left(\frac{A^n+B^n}{2}\right)^{1/n}\preceq X.
$$
Hence $A\vee B=\lim_{n\to\infty } \left(\frac{A^n+B^n}{2}\right)^{1/n}$.
\end{proof}

\vskip 5pt
\begin{cor}\label{sumvsmax} Let  $A,B\in\bM_d^+$. Then there exist unitary matrices $\{V_i\}_{i=1}^d$  in $\bM_d$ such that 
$$
A\vee B \le \frac{ 1}{d}  \sum_{i=1}^d V_i (A+ B) V_i^*.
$$
\end{cor}

\vskip 5pt
\begin{proof} By Theorem \ref{th2} with $g(t)=t$ and $f(t)=t^{1/n}$,
$$
(A^{n} + B^{n})^{1/n} \le \frac{ 1}{d}  \sum_{i=1}^d V_i (A+ B)V_i^*
$$
for some unitaries $V_i$. 
 Letting $n\to \infty$,  Kato's characterization of the  supremum and the compactness of the unitary group complete the proof.
\end{proof}

\vskip 5pt To introduce a  conjecture to be stated soon,  we formulate the following equivalent form of  Corollary \ref{sumvsmax}.

\vskip 5pt
\begin{cor} Let $A,B\in\bM_d^+$. Then, for all symmetric norms,
\begin{equation}\label{normmax}
 \| A\vee B\|\le \| A+ B\|,
\end{equation}
equivalently,
\begin{equation}\label{submax} 
 A\vee B \prec_{w} A+ B.
\end{equation}
\end{cor}

\vskip 5pt
For sake of completness, we note that Ando \cite[Lemma 6.15]{Ando} observed the infimum counterpart of Kato's theorem:

\vskip 5pt
\begin{prop}  Let $A,B\in\bM_d^+$ be invertible. Then 
$$
A\wedge B =\lim_{n\to\infty} \left(A^{-n}+B^{-n}\right)^{-1/n}.
$$
\end{prop}

\vskip 5pt
This is an immediate consequence of Kato's theorem combined with the fact that the inverse reverses the spectral order. For positive definite matrices,
$$
Y\preceq X\preceq A, B \iff A^{-1}, B^{-1} \preceq X^{-1}\preceq Y^{-1}
$$
Consequently $A\wedge B= (A^{-1}\vee B^{-1})^{-1}=\lim_{n\to\infty} \left(A^{-n}+B^{-n}\right)^{-1/n}$.

\subsection{Comments}

The submajorization relation \eqref{submax} complements the following supermajorization inequality due to Ando \cite[Theorem 6.16]{Ando}:
$$A\vee B\prec^{w} A^{\downarrow} \vee B^{\downarrow}.$$
On the other hand, the supermajorization 
 $A\vee B \prec^{w} A+B$ fails in general, already when $A=B$.

It would be highly desirable to extend this results to operators on infinite dimensional Hilbert spaces, possibly unbounded. Readers familiar with noncommutative integration theory in semifinite von Neumann algebras may find the following conjecture plausible.

\begin{conjecture}
The submajorization \eqref{submax} holds  in the positive cone $\mathcal{P}$ of the space of all $\tau$-measurable operators 
affiliated with a semifinite von Neumann algebra endowed with a faithful normal semifinite trace $\tau$.
Thus \eqref{normmax}  holds  for any norm respecting the submajorization in  $\mathcal{P}$.
\end{conjecture}

Coming back to the finite dimensional setting, we define for $Z\in\bM_d$ its maximal symmetric modulus and minimal symmetric modulus as, respectively,
$$
|Z|_{\vee} :=|Z|\vee |Z^*|, \quad |Z|_{\wedge} :=|Z|\wedge |Z^*|.
$$
A few  questions then occur.

\begin{question} Let $A,B\in\bM_d$. Does there exist a constant $c_{\vee}(d)$, depending only on $d$, such that  
$$
|A+B|_{\vee} \le c_{\vee}(d) \left(U|A|_{\vee} U^* + V|B|_{\vee}V^*\right)
$$
for some unitary matrices $U,V$ ? Can we choose $c_{\vee}(d)$ independent of $d$ ? 
 Can we take $c_{\vee}(d)=1$ ? The latter holds when $A$ and $B$ are commuting positive matrices.
\end{question}

\begin{conjecture} For every dimension $d$, one may choose $c_{\vee}(d)=1$.
\end{conjecture}

If one replaces the maximal modulus $|\cdot|_{\vee}$ by the minimal one $|\cdot|_{\wedge}$, the answer to the corresponding questions is negative. Letting
$$
A=\begin{pmatrix} 0&1 \\ 0&0\end{pmatrix}=B^*
$$
we have $|A+B|_{\wedge}=I$ while $|A|_{\wedge}=|B|_{\wedge}=0$.

We recall  a  related,  strong conjecture proposed by Zhang \cite{Z2}:

\begin{conjecture} If $A,B\in\bM_d$, then there exist two unitaries $U,V\in\bM_d$ such that
$$
|A+B|_{\qsym} \le \sqrt{2} \left(U|A|_{\qsym} U^* + V|B|_{\qsym}V^*\right)
$$
\end{conjecture}

\section{Orthogonal orbits}

An orthogonal matrix is a real unitary matrix. To extend our results to real matrices and orthogonal orbits, we first need the following simple observation.
For a positive integer $d$, we denote by $\omega(d)$ the smallest dyadic number ($\omega(d)=2^m$ for some integer $m$) such that $d\le \omega(d)$.

\vskip 5pt
\begin{lemma}\label{real}  There exist  orthogonal, diagonal matrices  $\{R_i\}_{i=1}^{\omega(d)}$  in $\bM_d$ such that 
$$
\Delta(X) =\frac{1}{\omega(d)}\sum_{i=1}^{\omega(d)} R_iXR_i^*
$$
for every matrix $X\in\bM_d$.
\end{lemma}

\vskip 5pt
Since the matrices $R_k$
 are diagonal with entries $\pm1$, we have $R_k^*=R_k$. For the proof, it is convenient to denote by $I_n$ the identity of size $n$.

\vskip 5pt
\begin{proof}  Let $Z\in\bM_{2^m}$, for an integer $m\ge 1$, be a real matrix partitioned as
$$
Z=\begin{bmatrix} Z_1&K \\ L&Z_2
\end{bmatrix}
$$
where $Z_1,Z_2\in\bM_{2^{m-1}}$. Then
$$
\begin{bmatrix} Z_1&0 \\ 0&Z_2
\end{bmatrix}=
 \frac{1}{2}\left\{ I_{2^m} Z I_{2^m} + \begin{bmatrix} I_{2^{m-1}}&0 \\ 0&-I_{2^{m-1}}\end{bmatrix} Z
\begin{bmatrix} I_{2^{m-1}}&0 \\ 0&-I_{2^{m-1}}\end{bmatrix}
\right\}.
$$
This proves the lemma when $m=1$. An induction on $m$ then yields the result for every dyadic dimension  $d=\omega(d)$.

Now, let $X\in\bM_d$ be a real matrix with $d<\omega(d)$, and define $Z\in\bM_{\omega(d)}$ as
$$
Z=\begin{bmatrix} X&0 \\ 0&0 \end{bmatrix}
$$
where the right lower $0$ stands for the zero matrix in $\bM_{\omega(d)-d}$. By the first step of the proof, the diagonal $\Delta(Z)$ of $Z$ can be expressed as
$$
\Delta\left(\begin{bmatrix} X&0 \\ 0&0 \end{bmatrix}\right)=
\frac{1}{\omega(d)} \sum_{k=1}^{\omega(d)} D_k\begin{bmatrix} X&0 \\ 0&0 \end{bmatrix}D_k^*
$$
for some orthogonal, diagonal matrices $D_k\in\bM_{\omega(d)}$. Let $R_k\in\bM_d$ be the diagonal, orthogonal matrix extracted from $D_k$ by taking its $d\times d$ left upper corner. The above equation yields
$$
\Delta(X)=
\frac{1}{\omega(d)} \sum_{k=1}^{\omega(d)} R_kXR_k^*
$$
and the proof is complete.
\end{proof}

\vskip 5pt
Replacing Bhatia's identity in Section 2 by Lemma \ref{real}, we obtain real analogues of our results, with orthogonal orbits replacing unitary ones. For instance, the real version of Lemma \ref{L1} reads as:

\vskip 5pt
\begin{lemma}\label{L3} Let $A,B\in\bM_d(\Omega)$ and let $g(t)$ be a nondecreasing convex function on $\Omega$. If $A$ and $B$ are real and  $A\prec_w B$, then 
there exist ortogonal matrices $\{U_i\}_{i=1}^{\omega(d)}$  in $\bM_d$ such that 
$$
g(A)\le \frac{1}{\omega(d)}\sum_{i=1}^{\omega(d)} U_i g(B)U_i^*.
$$
\end{lemma}

\vskip 5pt
If $g(t)$ is convex nondecreasing, $f(t)=-g(-t)$ is concave nondecreasing. Therefore, applying the real version of Lemma \ref{L2} together with the equivalence $A\prec B\iff -B\prec -A$, we have the following full-majorization counterpart of Lemma \ref{L3}:

\vskip 5pt
\begin{lemma}\label{L4} Let $A,B\in\bM_d(\Omega)$ and let $g(t)$ be a monotone convex function on $\Omega$. If $A$ and $B$ are real and  $A\prec B$, then 
there exist orthogonal matrices $\{U_i\}_{i=1}^{\omega(d)}$  in $\bM_d$ such that 
$$
g(A)\le \frac{1}{\omega(d)}\sum_{i=1}^{\omega(d)} U_i g(B)U_i^*.
$$
\end{lemma}

\vskip 5pt
The following corollary follows either from Lemma \ref{L4} or directly from Lemma \ref{L3}, since $A$ and $B$ have the same trace.

\vskip 5pt
\begin{cor}\label{correalprec} Let $A,B\in\bM_d(\Omega)$. If $A$ and $B$ are real and  $A\prec B$, then 
there exist orthogonal matrices $\{U_i\}_{i=1}^{\omega(d)}$  in $\bM_d$ such that 
$$
A= \frac{1}{\omega(d)}\sum_{i=1}^{\omega(d)} U_i BU_i^*.
$$
\end{cor}

\vskip 5pt
The next corollary is a special case:

\vskip 5pt
\begin{cor} Let $\Phi:\bM_d\to\bM_d$ be a positive linear map, unital and trace preserving.  If $S\in\bM_d$ is real symmetric, then there exist orthogonal matrices  $\{U_i\}_{i=1}^{\omega(d)}$  in $\bM_d$ such that 
$$
\Phi(S) =\frac{1}{\omega(d)}\sum_{i=1}^{\omega(d)} U_iSU_i^*.
$$
\end{cor}

\vskip 5pt
Indeed, it is well known that $\Phi(S)\prec S$. We may assume  $S\ge 0$ and 
the majorization follows readily from the standard variational formula, for $Z\in\bM_d^+$ and $1\le k\le d$,
$$
\sum_{j=1}^k \lambda_j^{\downarrow}(Z)=\min\left\{\mathrm{Tr\,}A +k \lambda_1^{\downarrow}(B)\right\}
$$
where the minmum is taken over all decompositions  $Z=A+B$ in $\bM_d^+$. See for instance \cite[Theorem 6.20]{HiP}. Another proof of $\Phi(S)\prec S$ is  \cite[Theorem 7.1]{Ando}.

We now reformulate, in the real setting, Corollary \ref{sumvsmax}. We state its multivariable version.

\vskip 5pt
\begin{cor}\label{sumvsmaxreal} Let  $\{A_k\}_{k=1}^m$ be real matrices in $\bM_d^+$. Then there exist  orthogonal matrices $\{U_i\}_{i=1}^{\omega(d)}$ in $\bM_d$  such that 
$$
\bigvee_{k=1}^m A_k \le \frac{1}{\omega(d)}  \sum_{i=1}^{\omega(d)} U_i \left\{\sum_{k=1}^mA_k\right\}U_i^*.
$$
\end{cor}

\vskip 5pt
The following remark
summarizes the history of the connection between majorization and unitary orbits.

\vskip 5pt
\begin{remark} It has long been known that, for Hermitian (real) matrices in $\bM_d$, the majorization $A\prec B$ is equivalent to the fact that $A$ belongs to the convex
hull of the unitary (orthogonal) orbit of $B$,
$$
A=\sum_{i=1}^N t_i U_i BU_i^*
$$
where the coefficients $t_j>0$ satisfy $\sum_j t_j= 1$.
See for instance \cite[Theorem 7.1]{Ando}, where we see that we can take $N=d!$ +1, as $d!$ is the cardinal of permutation matrices in $\bM_d$.  In fact one may assume that $A$ and $B$ are diagonal, and Carathéodory's theorem then shows that we can reduce   $N=d+1$. Somewhat surprisingly, this observation seems not to have appeared before Zhan's 2003 note \cite{Zhan}. We have seen that   we can actually take $N=d$, in case of a unitary orbit, and further  obtain an average in the unitary orbit; i.e.,  equality of all $t_j$. For averages over orthogonal orbits, we have shown that one may take $N=\omega(d)$ . The final questions below ask whether this bound is optimal.
\end{remark}

\begin{question}\label{qf1} What is the smallest integer $\varphi(d)$ such that there exist orthogonal matrices  $\{U_i\}_{i=1}^{\varphi(d)}$ in $\bM_d$ satisfying
$$
\Delta(X) =\frac{1}{\varphi(d)}\sum_{i=1}^{\varphi(d)} U_iXU_i^*.
$$
for every real matrix $X\in\bM_d$ ?
Is it $\varphi(d)=d$  or $\varphi(d)=\omega(d)$ ?
\end{question}

\begin{question}\label{qf2} The same as Question \ref{qf1}, but with orthogonal matrices $U_i$ that may depend on $X.$
\end{question}

\vskip 15pt
\noindent
Jean-Christophe Bourin

\noindent
Université Marie et Louis Pasteur, CNRS, LmB (UMR 6623), F-25000 Besançon, France.

\noindent

\noindent

\noindent
Email: jcbourin@univ-fcomte.fr

  \vskip 15pt\noindent
Eun-Young Lee

\noindent
Department of mathematics, KNU-Center for Nonlinear Dynamics,

\noindent
Kyungpook National University,

\noindent
Daegu 702-701, Korea.

\noindent
 Email: eylee89@knu.ac.kr


\begin{thebibliography}{99}
{\small

 \bibitem{AW} 
 C. Akemann and N. Weaver: Minimal upper bounds of commuting operators,
Proc. Amer. Math. Soc. 124, 3469–3476 (1996).

\bibitem{Ando} 
T.\ Ando, Majorization, doubly stochastic matrices, and comparison of eigenvalues, Lin. Alg.
Appl. 118 (1989), 163-248.

 \bibitem{Bh-monthly}  R.\ Bhatia, Pinching, trimming, truncating, and averaging of matrices. {\it Amer.\ Math.\ Monthly}  107  (2000),  no. 7, 602--608. 



\bibitem{Bpams} J.-C.\ Bourin,  A matrix subadditivity inequality for symmetric norms, {\it Proc. Amer. Math. Soc.}\ 138 (2010), no. 2, 495–504.

\bibitem{BH} J.-C.\ Bourin and F.\ Hiai,
 Jensen and Minkowski inequalities for operator means and anti-norms. {\it Linear Algebra Appl.}\ 456 (2014), 22–53. 


\bibitem{BLjot} J.-C.\ Bourin and E.-Y.\ Lee, Concave functions of positive operators, sums, and congruences. J. Operator Theory 63 (2010), no. 1, 151–157.


\bibitem{BL}
J.-C.\ Bourin and E.-Y.\ Lee, Unitary orbits of Hermitian operators with convex
or concave functions, {\it Bull. London Math. Soc.} {\bf 44} (2012),
1085--1102.

\bibitem{BLtv} J.-C.\ Bourin and E.-Y. Lee, 
 Matrix inequalities from a two variables functional, {\it Internat.\ J.\ Math.}\ 27 (2016), no.\ 9, , 1650071, 19 pp.


\bibitem{BLsym} J.-C.\ Bourin and E.-Y.\ Lee, Matrix parallelogram laws and symmetric moduli, {\it Internat.\ J.\ Math.} 37
(2026), no.\ 3, 2650018.

\bibitem{BLtriang} J.-C.\ Bourin and E.-Y.\ Lee, Triangle inequalities for the operator symmetric modulus, preprint, arXiv:2602.19607

\bibitem{Chan} N. N. Chan and  Kim Hung Li, Diagonal elements and eigenvalues of a real symmetric matrix, {\it J. Math. Anal. Appl.} 91 (1983), no. 2, 562–566.

\bibitem{HiP} F.\ Hiai, D.\ Petz, Introduction to matrix analysis and applications. Universitext. Springer, Cham; Hindustan Book Agency, New Delhi, 2014.

\bibitem{HJS}
B.\  T.\ Hoai, C.\ R.\ Johnson, I.\ M.\ Spitkovsky, Spectral dominance and commuting chains, {\it Proc. Amer. Math. Soc.}\ 136 (2008), no. 6, 2019–2029. 


\bibitem{K} 
T. Kato, Spectral order and a matrix limit theorem, {\it Linear  Mult. Alg.} 8 (1979),
15-19.

\bibitem{M}  L.\ Mirsky,  Matrices with prescribed characteristic roots and diagonal elements,  {\it J.\ London Math.\ Soc.}\ 33 (1958), 14–21.

\bibitem{Ols} M.\ P.\ Olson, The selfadjoint operators of a von Neumann algebra form a conditionally complete
lattice, {\it Proc. Amer. Math. Soc.} 28 (1971), 537-544.


\bibitem{Rot} S.\ Ju.\ Rotfel'd,  The singular values of a
sum of completely continuous operators,   Topics in
Mathematical Physics, Consultants Bureau, Vol.\ 3 (1969)
73-78.

\bibitem{T}   R.\ C.\ Thompson, Convex and concave functions of singular
values of matrix sums, {\it Pacific J. Math.}\  66 (1976), 285-290.



\bibitem{Zhan} X.\ Zhan, The sharp Rado theorem for majorizations, {\it Amer.\ Math.\ Monthly} 110
(2003) 152--153.

\bibitem{TZsym} T.\ Zhang, An operator triangle inequality for the quadratic symmetric modulus, preprint, arXiv:2602.01463.

\bibitem{Z2} T.\ Zhang, Operator symmetric moduli and sharp triangle inequalities, preprint,  arXiv:2603.01046.
}

\end{thebibliography}
\end{document}